\documentclass[11pt,reqno,a4paper]{amsart}

\usepackage[utf8]{inputenc}
\usepackage{graphicx}
\usepackage{color}
\usepackage{transparent}

\usepackage{latexsym}
\usepackage{amsmath}
\usepackage{amsfonts}

\newcommand{\D }{\Delta }

\newcommand{\cqfd}{$\hfill{\Box}$}

\renewcommand{\O }{\Omega }

\newcommand{\ba}{\begin{align*}}
\newcommand{\ea}{\end{align*}}

\numberwithin{equation}{section}
 \usepackage{epsfig,color,xcolor}
 
 \newtheorem{theorem}{Theorem}[section]
\newtheorem{proposition}[theorem]{Proposition}
\newtheorem{lemma}[theorem]{Lemma}

\newtheorem{remark}[theorem]{Remark}



\newcommand{\ble}{\begin{lemma}}
\newcommand{\ele}{\end{lemma}}
\newcommand{\be}{\begin{equation*}}
\newcommand{\ee}{\end{equation*}}

\newcommand{\bel}{\begin{equation}}
\newcommand{\eel}{\end{equation}}

\newcommand{\ep}{\varepsilon}
\newcommand{\fr}{\frac }

\newcommand{\R}{\mathbb{R}}

\renewcommand{\to}{\rightarrow}
\newcommand{\To}{\longrightarrow}

\newcommand{\upp}{u_p}

\def\sideremark#1{\ifvmode\leavevmode\fi\vadjust{\vbox to0pt{\vss% the remark
 \hbox to 0pt{\hskip\hsize\hskip1em%                          will appear only
 \vbox{\hsize2.1cm\tiny\raggedright\pretolerance10000%          on the side
  \noindent #1\hfill}\hss}\vbox to15pt{\vfil}\vss}}}%

\numberwithin{equation}{section}
\begin{document}
\title[]{Convergence for a planar elliptic problem \\ with large exponent Neumann data}

\author[]{ Habib Fourti}
\address{Address: University of Monastir, Faculty of Sciences, 5019 Monastir, Tunisia.}
\address{ $\quad \quad \quad \quad \ $Laboratory LR 13 ES 21, University of Sfax, Faculty of Sciences,
3000 Sfax, Tunisia.}
\address{e-mail: habib.fourti@fsm.rnu.tn - habib40@hotmail.fr}



\maketitle
\begin{abstract}
We study positive solutions $u_p$ of the nonlinear Neumann elliptic
problem $\D u =u$ in $\Omega $, $\partial u/\partial\nu =
|u|^{p-1}u$ on $\partial\O$, where $\Omega $ is a bounded open
smooth domain in $\R^2$. We investigate the asymptotic behavior of
families of solutions $u_p$ satisfying an energy bound condition
when the exponent $p$ is getting large. Inspired by the work of
Davila-del Pino-Musso \cite{DavilaDM}, we prove that $u_p$ is
developing $m$ peaks $x_i\in\partial \Omega$, in the sense
$u_p^p/\int_{\partial \Omega}u_p^p$ approaches the sum of $m$ Dirac
masses at the boundary and we determine the localization of these
concentration points.
%that $pu_p\rightharpoonup 2\pi e\sum^k_{i=1}\delta_{x_i} $, as $p\rightarrow \infty$.
 %More precisely we prove
 %that $pu_p$ approaches the sum of $k$ Dirac masses at the boundary.
\end{abstract}

\section{Introduction} \label{SectionIntroduction}
The purpose of this paper is to study the following problem
\begin{equation}\label{problem1}
\left\{\begin{array}{lr}\Delta u= u\qquad  \mbox{ in }\Omega\\
\displaystyle\frac{\partial u}{\partial \nu}=|u|^{p-1}u\ \mbox{ on
}\partial \Omega
\end{array}\right.
\end{equation}
where $p>1$, $\Omega\subset\R^2$ is a smooth bounded domain and
$\partial/\partial \nu$ denotes the derivative with respect to the
outward normal to $\partial \Omega$. Elliptic problem with nonlinear
boundary condition has been widely studied in the past by many
authors and it is still an area of intensive research, see for
instance \cite{BFS, Castro2, Castro1, Cherrier, DavilaDM, GL, LWZ,
Takahashi}.\\
Problem \eqref{problem1} has a variational structure. Indeed, its
solutions are in a one-to-one correspondence with the critical
points of the functional:
$$E_p(u)=\frac{1}{2}\int_{\Omega}|\nabla u|^2+ u^2 \
 dx- \frac{1}{p+1}\int_{\partial \Omega}|u|^{p+1} \ d\sigma(x)$$
defined on the sobolev space $H^1(\Omega)$. Trace and Sobolev
embeddings tell us that we have
$$H^1(\Omega)\hookrightarrow
H^\frac{1}{2}(\partial\Omega)\hookrightarrow L^p(\partial\Omega)$$
and that the embeddings are compact for every $ p >1$. Since in
dimension $2$, any exponent $p>1$ is subcritical (with respect to
the Sobolev embedding) it is well known, by standard variational
methods, that \eqref{problem1} has at least one positive solution.\\
Our main result provides a description of the asymptotic behavior,
as $p\rightarrow +\infty$, of positive solutions of \eqref{problem1}
under a uniform bound of their energy, namely we consider any family
$(u_p)$ of positive solutions to \eqref{problem1} satisfying the
condition
\begin{equation}\label{energybound}
p\displaystyle\int_{\Omega}|\nabla u_p|^2+u_p^2 \ dx \rightarrow
\beta \in \mathbb{R}, \quad \hbox{as }p\rightarrow +\infty.
\end{equation}
Our strategy goes along the method developed by Davila, del Pino and
Musso in \cite{DavilaDM} when they analyzed a nonlinear exponential
Neumann boundary condition. Indeed, Davila et al. were interested to
the following problem
\begin{equation}\label{problem11}
\left\{\begin{array}{lr}\Delta u= u\qquad  \mbox{ in }\Omega\\
\frac{\partial u}{\partial \nu}=\varepsilon e^u\ \mbox{ on }\partial
\Omega
\end{array}\right.
\end{equation}
where $\varepsilon$ is a small parameter. They proved that any
family of solutions $u_\varepsilon$ for which
$\varepsilon\int_{\partial \Omega}e^{u_\varepsilon}$ is bounded
develops, up to subsequences, a finite number $m$ of peaks $\xi_i\in
\partial\Omega$, $\varepsilon\int_{\partial
\Omega}e^{u_\varepsilon}\rightarrow 2m\pi$, and reciprocally, they
established that at least two such families exist for any given
$m\geq 1$.

 $\quad$ There is another source of
motivation for problem we are considering here. Its analogous usual
elliptic equation is
\begin{equation}\label{problem111}
\left\{\begin{array}{lr}\Delta u= |u|^{p-1}u\qquad  \mbox{ in }\Omega\\
 u=0\ \mbox{ on }\partial \Omega
\end{array}\right.
\end{equation}
known as Lane-Emden equation. Such equation has been investigated
widely in the last decades, see for example
\cite{AdiGrossi,DGIP1,DeMarchisIanniPacellaJEMS,
DeMarchisIanniPacellaPositivesolutions,DeMarchisIanniPacellaLondon,GrossiGrumiauPacella1}.
%%%%%%%%%%%%%%%%%%%%%%%%%%%%%%%%%%%
Concerning general positive solutions (i.e. not necessarily with
least energy) of the Lane-Emden Dirichlet problem, a first
asymptotic analysis was carried out in
\cite{DeMarchisIanniPacellaJEMS} showing that, under the
corresponding energy bound condition, all solutions $(u_p)$
concentrate at a finite number of points in $\overline{\Omega}$.
Later the same authors gave in
\cite{DeMarchisIanniPacellaPositivesolutions} a description of the
asymptotic behavior of $u_p$ as $p\rightarrow \infty$. They
completed this study in a recent work with Grossi \cite{DGIP1}. More
precisely, they showed quantization of the energy to multiples of
$8\pi e$ and proved convergence to $\sqrt{e}$ of the $L^\infty$-
norm, thus confirming the conjecture made in
\cite{DeMarchisIanniPacellaPositivesolutions}. A proof of this
quantization conjecture was also independently done by Thizy
\cite{Thizy}.

%It is natural to ask
%whether a similar phenomenon appears for the case of
 %nonlinear Neumann problem \eqref{problem1}.\\
 %%%%%%%%%%%%%%%%%%%%%%
$\quad$ Going back to \eqref{problem1}, the asymptotic behavior of
general positive solutions has not been studied yet. Before stating
our theorem let us review some known facts. In \cite{Takahashi},
Takahashi studied \eqref{problem1} by analyzing the asymptotic
behavior of least energy solutions (hence positive), as
$p\rightarrow \infty$. He proved that the least energy solutions
remain bounded uniformly with respect to $p$ and develop one peak on
the boundary. The location of this blow-up point is associated with
a critical point of the Robin function $H(x, x)$ on the boundary,
where $H$ is the regular part of the Green function of the
corresponding linear Neumann problem. More precisely, the Green
function $G(x, y)$ is the solution of the problem
\begin{equation}\label{Greenequation}
\left\{\begin{array}{lr}\Delta_x G(x,y)= G(x,y)\qquad  \mbox{ in }\Omega,\\
\displaystyle\frac{\partial G}{\partial \nu_x}(x,y)=\delta _y(x)
\mbox{ on }\partial \Omega,
\end{array}\right.
\end{equation}
for all $y\in \partial \Omega$ and its regular part
\begin{equation}\label{regularpart}
H(x, y) = G(x, y) - \frac{1}{\pi}\log \frac{1}{ |x - y|} .
\end{equation}
%%%%%%%%%%%%%%%%%%%%%%%%%%%%%%%%%%%%
%The first paper performing an asymptotic analysis of
%\eqref{problem1}, as $p\rightarrow \infty$, is \cite{Takahashi}
%where the author proves a 1-point concentration phenomenon for least
%energy (hence positive) solutions to \eqref{problem1} and derives
%some asymptotic estimates by using ideas of Ren and Wei (see
%\cite{RenWeiTAMS1994,RenWeiPAMS1996}).
Note that least energy solutions $(u_p)$ of this 2-dimensional
semi-linear Neumann problem satisfy the condition
$$p\displaystyle\int_{\Omega}|\nabla u_p|^2+u_p^2\ dx \rightarrow
2\pi e, \quad \hbox{as }p\rightarrow +\infty,$$ which is a
particular case of \eqref{energybound}. Later, following a similar
argument firstly introduced in \cite{AdiGrossi}, Castro
\cite{Castro1} identified a limit problem by showing that suitable
scaling of the least energy solutions $(u_p)$ converges in $C^1_
{loc}(\overline{\mathbb{R}^2_+})$ to a regular solution $U$ of the
Liouville problem
\begin{equation}\label{Liouvilleproblem}
\left\{\begin{array}{lr}\Delta U= 0\qquad  \mbox{ in }\mathbb{R}^2_+\\
\displaystyle\frac{\partial U}{\partial \nu}=e^U\ \mbox{ on
}\partial
\mathbb{R}^2_+\\
\int_{\partial\mathbb{R}^2_+} e^{U}=2\pi  \hbox{ and }
\sup_{\overline{\mathbb{R}^2_+}} U<\infty.
\end{array}\right.
\end{equation}
He also proved that $\|u_p\|_{\infty}$ converges to $\sqrt{e}$ as
$p\rightarrow +\infty$, as it had been previously conjectured in
\cite{Takahashi}. All these results are respectively similar to
those contained in \cite{RenWeiTAMS1994,RenWeiPAMS1996} and
\cite{AdiGrossi} which focus on the least energy solution of the
Lane-Emden problem in the plane.\\
However, problem \eqref{problem1} may have positive solutions with
an arbitrarily large number of boundary peaks, as shown by Castro in
\cite{Castro1}. Indeed, he proved that given any integer $m\geq 1$,
problem \eqref{problem1} has at least two families of positive
solutions $u_p $, each of them satisfying
$$pu_p(x)^{p+1}\rightharpoonup   \displaystyle 2\pi e\sum^m_{i=1}\delta_{\xi_i}
\quad \hbox{weakly in the sense of measure in }
\partial \Omega,
$$
as $p\rightarrow +\infty$, and the peaks of these two solutions are
located near points $\xi=(\xi_i, \ldots, \xi_m)\in (\partial
\Omega)^m$ corresponding to two distinct critical points of the
following functional defined on $(\partial \Omega)^m$
$$\varphi_m(x_1, \ldots, x_m) :=-\big[ \displaystyle\sum_{i=1}^m
 H(x_i , x_i ) + \sum_{j\neq i} G(x_i , x_j ) \big].$$
 %In \cite{Castro2}, Castro constructed a family
%of positive solutions
 %of problem \eqref{problem1} which develops exactly m peaks $\xi_i\in\partial\Omega$, in the sense
%that $pu_p \rightharpoonup2\pi e\sum^m_{i=1}\delta_{\xi_i}$, as
%$p\rightarrow\infty$.
It is natural to ask whether these properties hold for all families
of positive solutions $(u_p)$ satisfying \eqref{energybound}, as
$p\rightarrow \infty$. To our knowledge, a complete answer to this
conjecture has not been given so far, while partial results are
available as we describe below. In fact we extend the concentration
result in \cite{DavilaDM}, concerning elliptic problem with
exponential Neumann data, to a large exponent one. More precisely,
we prove that $u_p^p/\int_{\partial \Omega}u_p^p$ approaches the sum
of $m$ Dirac masses at the boundary. The location of these possible
points of concentration may be further characterized as solutions of
a system of equations defined explicitly in terms of the gradients
of the
above Green function and its regular part.\\
In order to state our main result we introduce some notations. Let
$$v_p=\displaystyle\frac{u_p}  {\int_{\partial \Omega}u_p^p \
d\sigma(x)},$$ where $u_p$ is a positive solution of
\eqref{problem1} satisfying \eqref{energybound}. We define the
blow-up set $S$ of $v_{p_n}$ to be the subset of $\partial\Omega$
such that $x \in S$ if there exist a subsequence, still denoted by
$v_{p_n}$, and a sequence $x_n$ in $\overline{\Omega}$ with
$$v_{p_n}(x_n) \rightarrow +\infty \quad \hbox{and}\quad
x_n\rightarrow x .$$
%Let, for each $1\leq i \leq m$,
 %where the function $\widetilde{g}_n$ and the
%radius $R$ are defined later.\\
 Now, we are able to state the following result:
\begin{theorem}\label{theorem}
Let $\Omega\subset\mathbb{R}^2$ be a smooth bounded domain. Then for
any sequence $v_{p_n}$ of $v_p$ with $p_n\rightarrow \infty$, there
exists a subsequence (still denoted by $v_{p_n}$) such that the
following statements hold true.
\begin{enumerate}
\item There exists a finite collection of distinct points $x_i \in
\partial \Omega,\quad i = 1,\ldots,m$ such that $S=\{x_i, \ 1\leq i\leq
m\}$.
\item $$ f_n:=\frac{u_{p_n}^{p_n}}{\int_{\partial \Omega}u_{p_n}^{p_n} \
d\sigma(x)}\displaystyle\rightharpoonup^*
\sum^m_{i=1}a_i\delta_{x_i}$$in the sense of Radon measures on
$\partial \Omega$ where
\begin{equation}\label{weight}
a_i=\displaystyle\lim_{r\rightarrow 0}\lim_{n\rightarrow
+\infty}\left(\frac{2\pi}{(p_n+1)(\int_{\partial
\Omega}u_{p_n}^{p_n}d\sigma(x))^2}\displaystyle\int_{\partial
\Omega\cap B_r(x_i)}u_{p_n}^{p_n+1} d\sigma(x) \right)^\frac{1}{2},
\quad \forall 1\leq i \leq m.\end{equation}
\item $\displaystyle v_{p_n}\rightarrow \sum^m_{i=1}a_iG(. , x_i)$ in $C^1_{loc}(\overline{\Omega}\setminus S)$, $L^t(\Omega)$ and
$L^t(\partial\Omega)$ respectively for any $1\leq t < \infty$, where
$G$ is the Green's function for Neumann problem
\eqref{Greenequation}.
\item The concentration points $x_i, \
i=1,\ldots, m$ satisfy
\begin{equation}\label{x_j relazione}
a_i \nabla_{\tau(x_i)} H(x_i,x_i)+\sum_{i\neq
\ell}a_\ell\nabla_{\tau(x_i)} G(x_i,x_\ell)=0,
\end{equation}
where $\tau(x_i)$ is a tangent vector to $\partial\Omega$ at $x_i$.
\end{enumerate}
\end{theorem}
%%%%%%%%%%%%%%%%%%%%%%%%%%%%%%%%%%%%%%%%%%%%%%%%%
As we mentioned before to prove this result we will proceed as in
\cite{DavilaDM}. But, to adopt their argument, a blow up technique
is needed to get some useful estimates. Indeed in both
\cite{Castro1} and \cite{Takahashi} the authors established the
facts that
\begin{equation}\label{estimates}
c_1\leq\|u_p\|_{L^{\infty}(\overline{\Omega})}\leq c_2 \quad
\hbox{and}\quad c_3\leq p\int_{\partial\Omega} |u_p|^p d\sigma(x)
\leq c_4
\end{equation}
for some positive constants $c_1, c_2, c_3 \hbox{ and }c_4$
independent of $p$ for the case of least energy solutions to pursue
the analysis. In our case we prove that these estimates hold true
for general solutions (not necessarily positive ) of
\eqref{problem1} satisfying the bound energy condition
\eqref{energybound}.
 %%%%%%%%%%%%%%%%%%%%%%%%%
 We point out that this last condition is very crucial in our
 framework to analyze the asymptotic behavior of the families
 $(u_p)$. In fact, by using a suitable rescaling of the solution, we proved that the rescaled
 function about the maximum point of $|u_p|$, which is located on the boundary of $\Omega$, converges
 to the bubble not only for least energy solution (shown in \cite{Castro1}) but also
 for finite energy ones. This information allowed us to obtain \eqref{estimates}.
 This
will be the subject of Proposition \ref{teo:BoundEnergia} which is
the analogous result of \cite[Proposition
2.2]{DeMarchisIanniPacellaPositivesolutions} and \cite[Theorem 2.1]{
DeMarchisIanniPacellaLondon}
concerning Lane-Emden equation.\\
 %This two ingredients allowed us to get the result of the above theorem.\\
 %We should emphasize that the boundary blow up phenomenon does not
%exist in the case of Lane Emden equation.....\\
%%%%%%%%%%%%%%%%%%%%%%%%%%%%%%%%%%%%%%%%%%%%%
Let us point out that, as in \cite{DavilaDM} we have boundary
concentration phenomena due to the nonlinear condition. But the
exponent nonlinearity brings us some difficulties in our analysis.
%But we will encounter a difficulty that rises from the large
%exponent nonlinearity.\\
%This difficulty is overcame by a detailed local blow up analysis.\\
%%%%%%%%%%%%%%%%%%%%%%%%%%%%%%%%%%%%%%%%%%%%%%
\begin{remark}
In contrast with the exponential nonlinearity studied in
\cite{DavilaDM}, the argument of Davila et al does not give the
value of the coefficients or weights $a_i$'s nor the quantisation of
the energy result.
\end{remark}
To determine the data $a_i$'s we think that new ideas are needed.
May be a detailed local analysis is required to overcome this
difficulty. However arguing as in Brezis and Merle
\cite{Brezis-Merle}, we get $a_i\geq \pi/ L_0$ where $L_0$ is
defined in \eqref{L0}.\\
We conjecture that the $a_i$'s are equal and
more precisely we have
$$a_i= l^{-1}2\pi \sqrt{e}, \quad  \forall 1\leq i \leq m$$
where $l=\displaystyle\lim_{p\rightarrow +\infty} p\int_{\partial \Omega}u_p(x)^p$.\\
If we combine this conjecture with results of Theorem \ref{theorem}
we get, for any family of positive solutions $(u_p)$ of
\eqref{problem1} satisfying \eqref{energybound}, the following
results
%\begin{equation}
%\label{energylimitCappositive} p\int_{\Omega}|\nabla u_p|^2 +u_p^2 \
%dx\to C,\ \mbox{ as $p\to+\infty$ } C \geq 2\pi e.
%\end{equation}
\begin{itemize}
\item[$(i)$] up to subsequence
\[
pu_p(x)\to  2\pi\sqrt{e} \sum_{i=1}^m  G(x,x_i)\ \mbox{as} \
p\rightarrow +\infty,\  \mbox{in}\
C^1_{loc}(\bar\Omega\setminus\mathcal S),
\]
where $G$ is the Green's function for Neumann problem
\eqref{Greenequation};
\item[$(ii)$]
 $(x_1,\dots, x_m)$ is a critical point
of $\varphi_m$, that is the concentration points $x_i, \ i=1,\ldots,
m$ satisfy
\begin{equation}\label{x_j relazione}
 \nabla_{\tau(x_i)} H(x_i,x_i)+\sum_{\ell\neq
i}\nabla_{\tau(x_i)} G(x_i,x_\ell)=0.
\end{equation}
\end{itemize}
We also conjecture that $$\|u_p\|_{L^\infty(\overline{\Omega})}\to
\sqrt{e}\hbox{ and } p\int_\Omega |\nabla u_p(x)|^2+u_p^2(x)\,dx\to
m.2\pi e,\ \mbox{ as }p\to+\infty.$$
 This complete picture or behavior needs
more accurate analysis. Verification of these conjectures remains as
the future work \cite{DFIP}.\\

 %%%%%%%%%%%%%%%%%%%%%%%%%%%%%%%%%%%%%%%%%%%%%%%%%%%%%%%%%%%%%%%%%%%%
 %%%%%%%%%%%%%%%%%%%%%%%%%%%%%%%%%%%%%%%%%%%%%%%%%%%%%%%%%%
%Since the nonlinear term is on the boundary it is much more
%complicated to study the asymptotic behavior of the families $u_p$
%solutions of \eqref{problem1} than the usual nonlinear elliptic
%equation.\\

The remainder of this paper is organized as follows: Section 2 is
devoted to the asymptotic behavior of a general family $(u_p)$ of
nontrivial solutions of \eqref{problem1} satisfying
\eqref{energybound}. In Section 3 we give the proof of our theorem.%%%%%%%%%%%%%%%%%%%%%%%%%%%%%%%%%%%%%%%%%%%%%%%%%%%%%%%%%%%%%%
%%%%%%%%%%%%%%%%%%%%%%%%%%%%%%%%%%%%%%%%%%%%%%%%%%%%%%%%%%%%%%%%%%%%%%%
\section{General Asymptotic Analysis} \label{SectionGeneral Asymptotic Analysis}

\

It was first proved in \cite{LWZ} for more general nonlinearities,
that there exists at least one solution which changes sign. If the
nonlinearity is odd in $u$, as in our case, it is mentioned in
\cite{LWZ} that
 there exist infinitely many sign-changing solutions by a standard
argument (see the reference therein), so it makes sense to study the
properties of both positive and sign-changing solutions.
\\

This section is mostly devoted to the study of the  asymptotic
behavior of a general family $(u_p)_{p>1}$ of nontrivial solutions
of \eqref{problem1} satisfying the uniform upper bound
\begin{equation}
\label{energylimit} p\int_{\Omega}|\nabla u_p|^2 + u_p^2 \,dx\leq
C,\ \mbox{ for some  $C>0$ independent of $p$.}
\end{equation}

Recall that in \cite{Takahashi}  it has been proved that for any
family $(u_p)_{p>1}$ of nontrivial solutions of \eqref{problem1} the
following lower bound holds
\begin{equation}
\label{energylimitLower}
\liminf_{p\rightarrow+\infty}p\int_{\Omega}|\nabla u_p|^2
+u_p^2\,dx\geq 2\pi e,
\end{equation}
so the constant $C$ in \eqref{energylimit} is intended to satisfy
$C\geq 2\pi e$. Moreover if $u_p$ is sign-changing then we also know
that (see again \cite{Takahashi})
\begin{equation}
\label{energylimitLowerPosNeg}
\liminf_{p\rightarrow+\infty}p\int_{\Omega}|\nabla
u_p^{\pm}|^2+(u^{\pm}_p)^2 dx\geq 2\pi e.
\end{equation}

\

We recall that the energy functional associated to \eqref{problem1}
satisfies
\[
E_p(u)=\frac{1}{2}\|
u\|^2_{H^1(\Omega)}-\frac{1}{p+1}\|u\|_{L^{p+1}(\partial
\Omega)}^{p+1},\ \ u\in H^1(\Omega).
\]
Since for a solution $u$  of \eqref{problem1}
\begin{equation}
\label{energiaSuSoluzioni} E_p(u)=(\frac12-\frac1{p+1})\|
u\|^2_{H^1(\Omega)}=(\frac12-\frac1{p+1})\|u\|^{p+1}_{L^{p+1}(\partial\Omega)},
\end{equation}
then \eqref{energylimit}, \eqref{energylimitLower} and
\eqref{energylimitLowerPosNeg} are equivalent to uniform upper and
lower bounds for the energy $E_p$ or for the $L^{p+1}(\partial
\Omega)$-norm, indeed
\begin{eqnarray*}
& \displaystyle\limsup_{p\rightarrow +\infty}\  2 pE_p(u_p) =
\limsup_{p\rightarrow +\infty}\  p\int_{\partial\Omega} |u_p|^{p+1}
\, d\sigma(x) = \limsup_{p\rightarrow +\infty}\  p\int_{\Omega}
|\nabla u_p|^{2}+u_p^2 \, dx\leq C
\\
&\displaystyle \liminf_{p\rightarrow +\infty}\ 2
pE_p(u_p)=\liminf_{p\rightarrow +\infty}\ p\int_{\partial\Omega}
|u_p|^{p+1} \, d\sigma(x) = \liminf_{p\rightarrow +\infty}\
p\int_{\Omega} |\nabla u_p|^{2}+u_p^2 \, dx\geq 2\pi e
\end{eqnarray*}
and  if $u_p$ is sign-changing, also
\begin{eqnarray*}
&\displaystyle \liminf_{p\rightarrow +\infty}\ 2
pE_p(u_p^{\pm})=\liminf_{p\rightarrow +\infty}\
p\int_{\partial\Omega} |u_p^{\pm}|^{p+1} \, d\sigma(x) =
\liminf_{p\rightarrow +\infty}\ p\int_{\Omega} |\nabla
u_p^{\pm}|^{2}+(u^{\pm}_p)^2 \, dx\geq 2\pi e,
\end{eqnarray*}
we will use all these equivalent formulations throughout the paper.
\

\

Observe that by the assumption in \eqref{energylimit} we have that
\[E_p(u_p)\rightarrow 0, \ \ \|u_p\|_{H^1(\Omega)}\rightarrow 0,\ \  \mbox{as $p\rightarrow +\infty$}\]
\[E_p(u_p^{\pm})\rightarrow 0, \ \ \| u_p^{\pm}\|_{H^1(\Omega)}\rightarrow 0,\ \  \mbox{as
$p\rightarrow +\infty\qquad$ (if $u_p$ is sign-changing)}\] so in
particular $u_p^{\pm}\rightarrow 0$ a.e. as $p\rightarrow +\infty$.
\\
In this section, we will show that the solutions $u_p$ do not vanish
as $p\rightarrow +\infty$ (both $u_p^{\pm}$ do not vanish if $u_p$
is sign-changing) and  that moreover, differently with what happens
in higher dimension, they do not blow-up. The last information is a
consequence of the existence of the first bubble which is obtained
by the rescaling respect to the maximum point.
%claim is based on .
 A uniform upper and lower bounds of the quantity $p\int_{\partial \Omega}u_p^p\
 d\sigma(x)$ is also obtained (see Proposition \ref{teo:BoundEnergia} below).
 All these estimates are required to
 adopt the argument developed in \cite{DavilaDM}. Our key result is the following:
\begin{proposition}\label{teo:BoundEnergia}
Let $(\upp)$ be a family of solutions to \eqref{problem1} satisfying
\eqref{energylimit}. Then
\begin{itemize}

\item[\emph{$(i)$}] (No vanishing on the boundary).\\
\[\|u_p\|_{L^\infty(\partial \Omega)}^{p-1}\geq \lambda_1,\] where $\lambda_1=\lambda_1(\Omega)(>0)$ is
the first eigenvalue of the eigenvalue problem
\begin{equation}\label{eigenvalueproblem}
\left\{\begin{array}{lr}\Delta u= u\qquad  \mbox{ in }\Omega\\
\frac{\partial u}{\partial \nu}=\lambda u\ \mbox{ on }\partial
\Omega
\end{array}\right.
\end{equation}

defined on $H^1(\Omega)$.\\
 If $u_p$ is sign-changing then also $\|u_p^\pm\|_{L^\infty(\partial \Omega)}^{p-1}\geq \lambda_1$.
\\Moreover
\begin{equation}\label{soloMaggiore1InGenerale}
\liminf_{p\rightarrow +\infty}\|u_p\|_{L^\infty(\partial
\Omega)}\geq 1 \hbox{ and }\liminf_{p\rightarrow
+\infty}\|u_p\|_{L^\infty(\overline{\Omega})}\geq 1.
\end{equation}

\item[$(ii)$] (Existence of the first bubble).
Let $(x_p^+)_p\subset\partial\Omega$ such that
$|u_p(x_p^+)|=\|u_p\|_{L^\infty(\partial \Omega)}$. Let us set
\begin{equation}\label{muppiu}
\mu_{p}^+:=\left(p |\upp(x_p^+)|^{p-1}\right)^{-1}
\end{equation}
and  for $
t\in\widetilde{\Omega}_{p}^+:=\{t\in\overline{\R^2_+}\,:\,y_p+\mu_{p}^+t\in\Psi(\overline{\Omega}\cap
B_R(x_p))\}$
\begin{equation}\label{scalingMax}
z_{p}(t):=\fr{p}{\upp(x_p^+)}\bigg(\upp\big(\Psi^{-1}(y_p+\mu_{p}^+
t)\big)-\upp(x_{p}^+)\bigg),
\end{equation}
where $y_p=\Psi(x_p^+)$ and $\Psi$ is the change of coordinates
introduced in \eqref{changement}.\\
Then $\mu_{p}^+\to0$ as $p\to+\infty$ and
\[z_{p}\To U\mbox{ in }C^1_{loc}(\overline{\R^2_+})\mbox{ as }p\to+\infty\] where
\begin{equation}\label{v0}
U(t_1,t_2)=\log\left(\frac{4}{t_1^2+(t_2+2)^2}\right)
\end{equation}
is the solution of the Liouville problem \eqref{Liouvilleproblem}
satisfying $U(0)=0$.
\item[\emph{$(iii)$}] (No blow-up). There exists $C>0$ such that
\begin{equation}
\label{boundSoluzio}\|\upp\|_{L^\infty(\overline{\Omega})}\leq C,\
\mbox{ for all $p>1$.}
\end{equation}
\item[\emph{$(iv)$}] There exist constants $c,C>0$, such that for all $p$ sufficiently large we have
\begin{equation}
\label{boundEnergiap}c\leq p\int_{\partial\Omega} |u_p|^p d\sigma(x)
\leq C.
\end{equation}
\item[\emph{$(v)$}] $\sqrt{p}u_p\rightharpoonup 0$ in $H^1(\Omega)$ as $p\rightarrow +\infty$.
\end{itemize}
\end{proposition}
\begin{proof}
Point \emph{$(i)$} has been first proved for positive solutions in
\cite{Takahashi}, here we follow the proof in  \cite[Proposition
2.5]{GrossiGrumiauPacella1}. If $u_p$ is sign-changing, just observe
that  $u_p^{\pm}\in H^1(\Omega)$, where we know that
\[0<2\pi e-\varepsilon \overset{\eqref{energylimitLower}/\eqref{energylimitLowerPosNeg}}{\leq}
 \int_{\Omega}|\nabla u_p^{\pm}|^2+(u_p^{\pm})^2\,dx\overset{\eqref{energylimit}}{\leq} C<+\infty\] and
 that also by Poincar\'e inequality
 \begin{eqnarray*}
\int_{\Omega} |\nabla
u_p^{\pm}|^2+(u_p^{\pm})^2\,dx=\int_{\partial\Omega}|u_p^{\pm}|^{p+1}\,d\sigma(x)&\leq&
\|u_p^{\pm}\|_{L^\infty(\partial
\Omega)}^{p-1}\int_{\partial\Omega}|u_p^{\pm}|^2\,d\sigma(x)\\&\leq&
\frac{\|u_p^{\pm}\|_{L^{\infty}(\partial\Omega)}^{p-1}}{\lambda_1(\Omega)}\int_{\Omega}|\nabla
u_p^{\pm}|^2+(u_p^{\pm})^2\,dx.
\end{eqnarray*}
Hence $\|u_p^{\pm}\|_{L^\infty(\partial
\Omega)}^{p-1}\geq\lambda_1(\Omega)$ and
$\|u_p^{\pm}\|_{L^\infty(\overline{
\Omega})}^{p-1}\geq\|u_p^{\pm}\|_{L^\infty(\partial
\Omega)}^{p-1}\geq\lambda_1(\Omega)$.\\
If $u_p$ is not
sign-changing just observe that either $u_p=u_p^+$ or $u_p=u_p^-$
and the same proof as before applies.

\

The proof of \emph{$(ii)$} follows the same ideas in \cite{Castro1}
where the same result has been proved for least energy (positive)
solutions. In the sequel we will adopt the same method in \cite{BFS}
and \cite{Castro1} based on flattening the boundary near the maximum
point then using a classical blow up argument introduced in
\cite{AdiGrossi}. We start by the following crucial lemma.
\begin{lemma}\label{maximum}
Let $u_p$ be a nontrivial solution of \eqref{problem1} and the
sequence $(x_p^+)\in \overline{\Omega}$ s.t.
$|u_p(x_p^+)|=\|u_p\|_{L^{\infty}(\overline{\Omega})}$. Then
$x_p^+\in \partial \Omega$ for each integer $p\geq2 $.
\end{lemma}
\begin{proof}
We argue by contradiction. Suppose that there exists $p\geq 2$ such that $x_p^+\in\Omega$.\\
Recall that $x_p^+$ is a point where $|u_p|$ achieves its maximum.
Without loss of generality, we can assume that:
$$u_p(x_p^+)=\displaystyle\max_{\overline{\Omega}} u_p > 0.$$
Hence $x_p^+$ is an interior local maximum and $u_p(x_p^+)>0$. By
continuity of $u_p$, there exists $r>0$ such that $u_p(x)>0$, for
each $x\in B_r(x_p^+)$ and from \eqref{problem1} we get $\Delta u_p
>0$ in $B_r(x_p^+)$. Maximum principle implies that $u_p$ is a
constant function in $B_r(x^+_p)$. Therefore $u_p=\Delta u_p =0$ in $B_r(x_p^+)$ which contradicts $u_p(x_p^+)>0$.\\
If $u_p(x_p^+)<0$ then $x_p^+$ is a minimum and a similar argument
holds.
\end{proof}
Let $M_p=\max_{\overline{\Omega}}|u_p| $ and let $x^+_p$ be a
maximum point of $|u_p|$. Without loss of generality, we may assume
that $$u_p(x_p^+)=\displaystyle\max_{\overline{\Omega}} u_p
>0.$$ By $(i)$ we have that $pu_p(x_p^+)^{p-1}\rightarrow +\infty$ as $p\rightarrow
+\infty$, so \eqref{soloMaggiore1InGenerale} holds and moreover
$\mu_{p}^+\rightarrow 0$, where $\mu_p^+$ is defined in
\eqref{muppiu}.\\
From Lemma \ref{maximum}, we have $x^+_p \in
\partial \Omega$ and $\|u_p\|_{L^\infty(\overline{\Omega})}=\|u_p\|_{L^\infty(\partial{\Omega})}$. Up to a subsequence, $x_p^+$ converges to some
$\overline{x}\in \partial \Omega $. Assume that $\overline{x}$ is
located in the origin and the unit outward normal to $\partial
\Omega$ at $0$ is $(-e_2)$ where $e_2$ is the second element of a
canonical basis in $\mathbb{R}^2$. It will be convenient to work in
fixed half balls. For this reason, we
need some change of coordinates. This program was done in many works (see for example \cite{E} and \cite{RR}).\\
Since we will assume that $\partial \Omega$ is a $C^2$ surface, we
know that there is an $R> 0$ and a $C^2(\mathbb{R})$ function $\rho$
such that (after a possible renumbering and reorientation of
coordinates)
$$\partial \Omega \cap B_R(0)=\{x \in B_R(0): \quad x_2=\rho(x_1)\}$$
$$ \Omega \cap B_R(0) = \{x\in B_R(0): \quad x_2 >\rho(x_1)\}$$
and moreover, the mapping $$B_R(0) \ni x \mapsto y = \Psi(x)\in
\mathbb{R}^2$$ defined by
\begin{eqnarray}\label{changement}
\left\{\begin{array}{lll} y_1&:=&x_1,\\
y_2&:=&x_2-\rho(x_1),
\end{array}
\right.
\end{eqnarray}
is one-to-one. Define $\Phi := \Psi^{-1}$. Note that $\Psi$ is a
$C^2$ function that transforms the set $\Omega ' := \Omega\cap
B_R(0)$ (in what we refer to as $x$ space) into a set $\Omega "$ in
the half-space $y_n> 0$ (of $y$ space). Note also that the point
$\overline{x}=0$ is
mapped to the origin of $y$ space.\\
Our task now is changing the partial differential equation
\eqref{problem1} satisfied by $u_p$ in $\Omega '$
 into $y$ coordinates. We define
$$\widetilde{u}_p(y):= u_p(\Phi(y)), \quad \hbox{for all }y\in \Omega
" .$$ Let $\varphi \in \mathcal{D}(B_R(0)\cap \overline{\Omega})$.
Multiplying \eqref{problem1} by $\varphi$, integrating by part over
$B_R(0)\cap \Omega$ and using the change of variable $x=\Phi(y)$, we
find
\begin{align}\label{IBP}
\displaystyle &\int_{\Omega "}\nabla u_p(\Phi(y)).\nabla
\varphi(\Phi(y))+u_p(\Phi(y))\varphi(\Phi(y))dy\nonumber\\
&= \int_{\partial \Omega "\cap\partial \mathbb{R}_+^2}|
u_p(\Phi(y_1,0))|^{p-1}u_p(\Phi(y_1,0)) \varphi(\Phi(y_1,0))dy_1.
\end{align}
%where $B^+_R(0)=B_R(0)\cap\mathbb{R}^2_+$ and
%$D_R(0)=B_R(0)\cap\partial\mathbb{R}^2_+$ (the flat boundary of $B_{R}^+(0)$).\\
Let $\varphi_1(y):=\varphi(\Phi (y)), \ \hbox{for each }y\in \Omega
"$. A simple computation shows that
\begin{equation}\label{changementgradient} \nabla u_p(\Phi (y))=\nabla
\widetilde{u}_p(y)-\bigg(\rho'(y_1)\frac{\partial
\widetilde{u}_p}{\partial y_2}(y),0\bigg).
\end{equation}
The above relation holds also for
$\varphi$ and $\varphi_1$.\\
Using \eqref{IBP}, \eqref{changementgradient} and Green's formula,
we can prove that the
 functions $\widetilde{u}_p$ satisfy the following problem
\begin{equation}\label{chv}
\left\{\begin{array}{lll} \displaystyle \Delta
 \widetilde{u}_p-\widetilde{u}_p-2 \rho'(y_1)
  \frac{\partial ^2\widetilde{u}_p}{\partial y_1\partial y_2}-\rho''(y_1)
  \frac{\partial \widetilde{u}_p}{\partial y_2}+( \rho'(y_1))^2
   \frac{\partial ^2 \widetilde{u}_p}{\partial y_2^2}=0\quad \hbox{in }\Omega ",\\
\displaystyle\frac{\partial \widetilde{u}_p}{\partial \nu}+ \rho'
(y_1)
  \frac{\partial \widetilde{u}_p}{\partial y_2}-( \rho'(y_1))^2
\frac{\partial \widetilde{u}_p}{\partial
y_2}=|\widetilde{u}_p|^{p-1}\widetilde{u}_p\quad\hbox{on }\partial
\Omega "\cap\partial \mathbb{R}_+^2.
\end{array}
\right.
\end{equation}
Let $\overline{R}$ be such that $B_{\overline{R}}(0)\cap\{y |
y_2>0\}\subset\Omega"$ and define
$B^+_{\overline{R}}(0):=B_{\overline{R}}(0)\cap\mathbb{R}^2_+$ and
$D_{\overline{R}}(0):=B_{\overline{R}}(0)\cap\partial\mathbb{R}^2_+$
(the flat boundary of $B_{R}^+(0)$). In particular we can look at
problem \eqref{chv} as being defined only in the half-ball
$B^+_{\overline{R}}(0)$, that is
\begin{equation}\label{pbrbar}
\left\{\begin{array}{lll} \displaystyle \Delta
 \widetilde{u}_p-\widetilde{u}_p-2 \rho'(y_1)
  \frac{\partial ^2\widetilde{u}_p}{\partial y_1\partial y_2}-\rho''(y_1)
  \frac{\partial \widetilde{u}_p}{\partial y_2}+( \rho'(y_1))^2
   \frac{\partial ^2 \widetilde{u}_p}{\partial y_2^2}=0\quad \hbox{in }B_{\overline{R}}^+(0),\\
\displaystyle\frac{\partial \widetilde{u}_p}{\partial \nu}+ \rho'
(y_1)
  \frac{\partial \widetilde{u}_p}{\partial y_2}-( \rho'(y_1))^2
\frac{\partial \widetilde{u}_p}{\partial
y_2}=|\widetilde{u}_p|^{p-1}\widetilde{u}_p\quad\hbox{on
}D_{\overline{R}}(0).
\end{array}
\right.
\end{equation}
Now we perform a classical blow up argument. Let $p_0$ be a
sufficiently large integer such that $y_p=\Psi (x_p^+)\in
B_{\overline{R}/4}(0)$ for all $p>p_0$. Then we consider
\begin{equation*}
z_{p}(t):=\fr{p}{\widetilde{u}_p(y_p)}\big(v_p(y_p+\mu_{p}^+
t)-v_p(y_p)\big), \quad \forall t\in
\overline{B^+}_{\overline{R}/(2\mu_p^+)},\ \forall p>p_0
\end{equation*}
where $\widetilde{u}_p(y_p)=u_p(x_p^+)$. For simplicity we shall
write $x_p$ for $x_p^+$, $\mu_p$ for $\mu_p^+$, $B_{\overline{R}}^+$
for
$B_{\overline{R}}^+(0)$ and $D_{\overline{R}}$ for $D_{\overline{R}}(0)$. \\
The function $z_p$ satisfies the following system
\begin{equation}\label{normalisation}
\left\{\begin{array}{lll} \Delta
 z_p-\mu_p^2z_p-\mu_p^2p-2\displaystyle \rho'(\mu_p t_1+y_{p,1})
  \frac{\partial ^2z_p}{\partial t_1\partial t_2}\\
  \quad \displaystyle-\mu_p\rho''(\mu_p t_1+y_{p,1})
  \frac{\partial z_p}{\partial t_2}-(\rho'(\mu_p t_1+y_{p,1}))^2
   \frac{\partial ^2 z_p}{\partial t_2^2}=0\quad \hbox{in }B_{\mu_p^{-1}\overline{R}/2}^+,\\
\displaystyle\frac{\partial z_p}{\partial \nu}+
 \rho'(\mu_p t_1+y_{p,1})
  \frac{\partial z_p}{\partial t_2}\\
  \quad\displaystyle-(
\rho'(\mu_p t_1+y_{p,1}))^2 \frac{\partial z_p}{\partial
t_2}=|1+\frac{z_p}{p}|^{p-1}(1+\frac{z_p}{p})\quad\hbox{on
}D_{\mu_p^{-1}\overline{R}/2},\\
0<1+\frac{z_p}{p}\leq 1 \quad\hbox{and} \quad z_p(0)=0
\end{array}
\right.
\end{equation}
where $(t_1,t_2)$ is the coordinates of $t$ and $y_{p,1}$ is the
first component of $y_p$. $z_p$ satisfies the following equations
\begin{equation}\label{system}
\left\{\begin{array}{lll} -L_p z_p + \mu_p^2z_p =-\mu_p^2p
\quad \hbox{in }B_{\mu_p^{-1}\overline{R}/2}^+,\\
N_p z_p=|1+\frac{z_p}{p}|^{p-1}(1+\frac{z_p}{p})\quad\hbox{on
}D_{\mu_p^{-1}\overline{R}/2},
\end{array}
\right.
\end{equation}
where $L_p:=\Delta
 -2\displaystyle \rho'(\mu_p t_1+y_{p,1})
  \frac{\partial ^2.}{\partial t_1\partial t_2}
  \displaystyle-\mu_p\rho''(\mu_p t_1+y_{p,1})
  \frac{\partial .}{\partial t_2}-(\rho'(\mu_p t_1+y_{p,1}))^2
   \frac{\partial ^2 .}{\partial t_2^2}$ and
   $N_p:=\displaystyle\frac{\partial\  .}{\partial \nu}+
 \rho'(\mu_p t_1+y_{p,1})
  \frac{\partial \ .}{\partial t_2}\displaystyle-( \rho'(\mu_p t_1+y_{p,1}))^2 \frac{\partial\
.}{\partial t_2}$.
\begin{remark}\label{Remark}
Observe that $\rho'(0)=0$ and the continuity of $\rho''$ imply that
\begin{itemize}
\item$\displaystyle L_p\rightarrow_{p\rightarrow \infty} \Delta$,
\item $\displaystyle N_p\rightarrow_{p\rightarrow \infty} \frac{\partial \ .}{\partial
\nu}$.
\end{itemize}
\end{remark}
For fixed $r>0$ we consider $p_1>p_0$ large enough so that $8\mu_p
r<\overline{R} $ for all $p>p_1$, and consider the problem of
finding $w_p$ solution of
\begin{eqnarray}\label{w}
\left\{\begin{array}{rcl} -L_p w+\mu_p^2 w&=&-p\mu_p^2 \ \quad \quad \hbox{in }B_{4r}^+,\\
N_p w&=&(1+\frac{z_p}{p})^p\quad \hbox{on }D_{4r},\\
w&=&0\quad \ \quad \quad \quad\hbox{on }S_{4r},
\end{array}
\right.
\end{eqnarray}
where $S_{4r}=\partial B_{4r}\cap \mathbb{R}^2_+$ (the curved
boundary of $B_{4r}^+$). Firstly, the existence of such $w_p\in
H^1(B^+ _{4r})$ is guaranteed by Lax-Milgram theorem and it
satisfies
$$\|w_p\|_{H^1(B^+ _{4r})}\leq C\bigg(\|\mu_p^2 p\|_{L^2(B^+ _{4r})}
+\|(1+\frac{z_p}{p})^p\|_{L^2(D _{4r})}\bigg).$$ Moreover, observe
that for each $q \geq 2$, and all $p > p_1$
$$\displaystyle\int_{B^+_{4r}}|\mu_p^2p|^q dt \leq C$$
since we have $r\leq C \mu_p^{-1}$. Also
\begin{eqnarray*}
\displaystyle\int_{D_{4r}}|1+\frac{z_p}{p}|^{pq} d\sigma (t)&\leq&
\int_{D_{\mu_p^{-1}R/2}}|1+\frac{z_p}{p}|^{pq} d\sigma (t)\\
&=&
\mu_p^{-1}\int_{D_{R/2}(y_p)}\fr{|\widetilde{u}_p(y)|^{pq}}{\widetilde{u}_p(y_{p})^{pq}}d\sigma(y)\\&\leq&
\mu_p^{-1}\frac{1}{u_p(x_p)^{pq}}\int_{\partial
\Omega}|u_p(x)|^{pq}d\sigma(x)\\&\leq&
\frac{p}{u_p(x_p)^2}\int_{\partial
\Omega}|u_p(x)|^{p+1}d\sigma(x)\\
&\leq& C,
\end{eqnarray*}
where the last inequality holds from Proposition
\ref{teo:BoundEnergia} $(i)$ and \eqref{energylimit}. Hence using a
result from \cite{Shamir1} we conclude that when $q>4$, $w_p$ must
be in $W^{\frac{1}{2}+t,q}(B_{4r}^+)$ for $0<t<2/q$ with
\begin{equation}\label{estimationw}
\|w_p\|_{W^{\frac{1}{2}+t,q}(B_{4r}^+)}\leq C\bigg(\|\mu_p^2
p\|_{L^q(B^+ _{4r})} +\|(1+\frac{z_p}{p})^p\|_{L^q(D
_{4r})}\bigg)\leq C,
\end{equation}
where the constant $C$ is independent of $p$ since the coefficients
of the operator $(L_p,N_p)$ were uniformly bounded. Furthermore,
\eqref{estimationw} implies that $w_p$ is $L^\infty$ bounded. \\
Consider now the function $\varphi_p := w_p - z_p +
\|w_p\|_{L^\infty(B^+ _{4r})}$ which solves
\begin{eqnarray*}
\left\{\begin{array}{rcl} -L_p \varphi+\mu_p^2 \varphi&=&\mu_p^2
\|w_p\|_{L^\infty(B^+ _{4r})}\
\quad \quad \hbox{in }B_{4r}^+,\\
N_p \varphi&=&0\quad \quad \quad\hbox{on }D_{4r},\\
\varphi&\geq&0\quad \ \quad \quad \quad\hbox{in }B^+_{4r}.
\end{array}
\right.
\end{eqnarray*}
Note that, for p large, we have $N_p \varphi = 0$ is equivalent to
$\frac{\partial \varphi}{ \partial \nu }=0$ since the function
$(t_1,t_2)\mapsto\rho'(\mu_p t_1+y_{p,1})$ converges uniformly to
$0$. Hence the function $\varphi_p$ satisfies
\begin{eqnarray*}
\left\{\begin{array}{rcl} -L_p \varphi+\mu_p^2 \varphi&=&\mu_p^2
\|w_p\|_{L^\infty(B^+ _{4r})}\
\quad \quad \hbox{in }B_{4r}^+,\\
\frac{\partial \varphi}{ \partial \nu }&=&0\quad \quad \quad\hbox{on }D_{4r},\\
\varphi&\geq&0\quad \ \quad \quad \quad\hbox{in }B^+_{4r}.
\end{array}
\right.
\end{eqnarray*}
For $t=(t_1,t_2)\in B_{4r}$, we define the function
$$\hat{\varphi}_p=
\left\{\begin{array}{lll} \varphi_p(t)&\hbox{ if}&t_2\geq 0,\\
\varphi_p(t_1,-t_2)&\hbox{ if}&t_2< 0.
\end{array}
\right.
$$
Clearly $\hat{\varphi}_p$ is a non-negative solution of  $-L_p
\varphi+\mu_p^2 \varphi=\mu_p^2 \|w_p\|_{L^\infty(B^+ _{4r})}$ in
$B_{4r}$. Applying Harnack inequality (\cite[Theorem 4.17]{GT}), we
obtain for every $a\geq 1$
\begin{eqnarray*}
\left(\frac{1}{|B_{3r}|}\int_{B_{3r}}
\hat{\varphi}_p^a\right)^\frac{1}{a}&\leq& C\left\{\inf_{B_{3r}}
\hat{\varphi}_p+\left\|\mu_p^2 \|w_p\right\|_{L^\infty(B^+
_{4r})}\|_{L^2(B _{4r})}\right\}\\
&\leq& C \left\{\varphi_p(0)+\left\|\mu_p^2
\|w_p\right\|_{L^\infty(B^+ _{4r})}\|_{L^2(B _{4r})}\right\}\\
&\leq& C
\end{eqnarray*}
where we have used the facts that $z_p(0) = 0$ and $w_p$ is
uniformly bounded in $B_{4r}^+$.
%This implies
%that$\hat{\varphi}_p$ is uniformly bounded in $B_{ 3r}$.
 By interior
elliptic regularity (see for instance \cite[Theorem 9.13]{GT}) now
we obtain that
$$\|\hat{\varphi}_p\|_{W^{2,q}(B_{2 r})}\leq C \left( \left\|\mu_p^2 \|w_p\|_{L^\infty(B^+ _{4r})}\right\|_{L^q(B_{3r})}
+ \|\hat{\varphi}_p\|_{L^q(B_{3r})}\right)\leq C.$$
%By using a further transformation of coordinates we
%can map $\gamma(y)$ to $(0,-1)$ for all $y \in D_{4r}$, so that the
%resulting function can be extended across $s_2 = 0$,
%and also be a solution to an elliptic equation in $B_{3r}$ with
%smooth coefficients (with norms that can be bounded independently of $p$).
Hence, we get that
\begin{equation}\label{estimationvarphi}
 \varphi_p
 \hbox{ is uniformly bounded in }
W^{2,q}(B^+ _{2r}) \hbox{ for }q>1.\end{equation} It follows using
\eqref{estimationw} and \eqref{estimationvarphi} that
\begin{equation}\label{estimationz}
\|z_p\|_{W^{\frac{1}{2}+t,q}(B_{2r}^+)}\leq C
\end{equation}
for $q>4$, $0<t<2/q$ and any $p>p_1$.
 Finally, Shauder regularity will
tell us that $z_p$ is bounded in $C^{1,\alpha}(B^+_ r )$ for some $0
< \alpha < 1$, independently of $p> 1$ large. Thus by Arzela-Ascoli
Theorem and a diagonal process on $r\rightarrow \infty$, after
passing to a subsequence
\begin{equation}\label{convergence1}
z_p\rightarrow U \mbox{ in }  C^1_{loc} (\overline{\mathbb{R}^2_+})
\mbox{ as }p\rightarrow +\infty.
\end{equation}
Since $ \rho' (0)=0$ and $\mu_p\rightarrow0$ and by using Remark
\ref{Remark}, we conclude that $U$ satisfies the following problem
\begin{equation}\label{limitproblem1}
\left\{\begin{array}{lr}\Delta U= 0\qquad  \mbox{ in }\R^2_+\\
\frac{\partial U}{\partial \nu}=e^U\ \mbox{ on }\partial \R^2_+.
\end{array}\right.
\end{equation}
Moreover, we have $U(0)=0$ and $U\leq 0$. \\
In the sequel, we need the following result:
\begin{lemma}\label{boundness}
 Let $U$ satisfy \eqref{convergence1} and
\eqref{limitproblem1}, then we have
 $$\displaystyle\int_{\partial \mathbb{R}^2_+} e^U < \infty.$$
\end{lemma}
\begin{proof}
%%%%%%%%%%%%%%%%%%%%%%%%%%%%%%%%%%%%%%%%%%%%%%%%%%%%%%%%%%
For any $R>0$ and each $|t_1|<R$, we have
$$\displaystyle(p+1)[\log |1+\frac{z_p(t_1,0)}{p}|-\frac{z_p(t_1,0)}{p+1}]\longrightarrow_{p\rightarrow+\infty} 0 $$
So we can use Fatou's Lemma to write
\begin{eqnarray*}
\displaystyle\int_{-R}^Re^{U(t_1,0)}dt_1&\stackrel{\eqref{convergence1}\,+\,\textrm{{Fatou}}}{\leq}
&
\int_{-R}^Re^{z_p(t_1,0)-(p+1)[\log |1+\frac{z_p(t_1,0)}{p}|-\frac{z_p(t_1,0)}{p+1}]}dt_1+o_p(1)\\
&\leq&\int_{D_R(0)}|1+\frac{z_p(t)}{p}|^{p+1}d\sigma(t)+o_p(1)\\
&\leq&\int_{D_R(0)}\fr{|\widetilde{u}_p(y_{p}+\mu_{p}t)|^{p+1}}{\widetilde{u}_p(y_{p})^{p+1}}d\sigma(t)+o_p(1)\\
&\leq&\mu_p^{-1}\int_{D_{R\mu_{p}}(y_p)}\fr{|\widetilde{u}_p(y)|^{p+1}}{\widetilde{u}_p(y_{p})^{p+1}}d\sigma(y)+o_p(1)\\
&\leq&\fr p{\|u_p\|_{\infty}^2}\int_{\partial \Omega}|\upp(x)|^{p+1}d\sigma(x)+o_p(1)\\
&\stackrel{\eqref{soloMaggiore1InGenerale}}{\leq}&\fr
p{(1-\ep)^2}\int_{\partial\Omega}|\upp(x)|^{p+1}
d\sigma(x)+o_p(1)\stackrel{\eqref{energylimit}}{\leq} C<+\infty,
\end{eqnarray*}
so that $e^{U}\in L^1(\partial\R^2_+)$.
\end{proof}
 Recall that $U$ is a non positive solution of \eqref{limitproblem1}. Using Lemma \ref{boundness},
  $U$ satisfies the Liouville problem \eqref{Liouvilleproblem}.
  By virtue of the classification due to P. Liu \cite{Liu} (see also
  \cite[Theorem 1.3]{GL}), the solution $U$ must be of the form
$$U(t_1,t_2)=\log \frac{2\mu_2}{(t_1-\mu_1)^2+(t_2+\mu_2)^2},$$
for some $\mu_2
> 0$ and $\mu_1 \in \mathbb{R}$. Since $U(0)=0$ and $U\leq 0$, arguing as in \cite{Castro1} (see page 8) we obtain \eqref{v0}. Last an easy computation shows
that $\int_{\partial \mathbb{R}^2_+}e^{U}=2\pi$.

\

Point \emph{$(iii)$} has been first proved in \cite{Takahashi} using
same ideas contained in \cite{RenWeiTAMS1994}, here we write a
simpler proof which follows directly  from (ii) by applying Fatou's
lemma. An analogous argument can be found in \cite{Castro1} arguing
as in \cite[Lemma 3.1]{AdiGrossi}. Indeed, for each $p>p_0$ we have
\begin{eqnarray*}
\displaystyle\|u_p\|_{\infty}^2 2\pi=\|u_p\|_{\infty}^2
 \int_{\partial\R^2_+}e^{U(t)}d\sigma (t)&\overset{\mbox{\footnotesize{\emph{$(ii)$}-Fatou}}}{\leq}&
\|u_p\|_{\infty}^2\int_{D_{\mu_p^{-1}R}}\left|1+\frac{z_{p}(t)}{p}
 \right|^{p+1}d\sigma(t)\\
&\leq&\|u_p\|_{\infty}^2\int_{D_{\mu_p^{-1}R}}\fr{|\widetilde{u}_p(y_{p}+\mu_{p}t)|^{p+1}}{\widetilde{u}_p(y_{p})^{p+1}}d\sigma(t)\\
&\leq&\|u_p\|_{\infty}^2\mu_p^{-1}\int_{D_{R}(y_p)}\fr{|\widetilde{u}_p(y)|^{p+1}}{\widetilde{u}_p(y_{p})^{p+1}}d\sigma(y)\\
&\stackrel{\eqref{muppiu}}{\leq}&
p\int_{\partial\Omega}|\upp(x)|^{p+1}
d\sigma(x)\stackrel{\eqref{energylimit}}{\leq} C<+\infty,
\end{eqnarray*}
where $R$ is chosen such that $R\leq \frac{\overline{R}}{2}$. \

\

\emph{$(iv)$} follows directly from  \emph{$(iii)$}. Indeed on the
one hand
\[
0<C\overset{\eqref{energylimitLower}-\eqref{energiaSuSoluzioni}}{\leq}
p\int_{\partial\Omega} |u_p|^{p+1}\,
d\sigma(x)\leq\|u_p\|_{\infty}p\int_{\partial\Omega} |u_p|^p \,
d\sigma(x)\stackrel{\emph{$(iii)$}}{\leq}C p\int_{\partial\Omega}
|u_p|^p \, d\sigma(x)
\]
On the other hand by H\"older inequality
\[
p\int_{\partial\Omega} |u_p|^p\, d\sigma(x) \leq
|\partial\Omega|^{\frac{1}{p+1}} p\left(\int_{\partial\Omega}
|u_p|^{p+1}\,d\sigma(x)\right)^{\frac{p}{p+1}}\stackrel{\eqref{energylimit}}{\leq}C.
\]

\

\

To prove \emph{$(v)$} we need  \emph{$(iv)$}. Indeed let us note
that, since \eqref{energylimit} holds,
 there exists $w\in H^1(\Omega)$ such that, up to a subsequence,  $\sqrt{p}u_p \rightharpoonup  w$ in $H^1(\Omega)$.
  We want to show that $w=0$ a.e. in $\Omega$.

Using the equation \eqref{problem1}, for any test function
$\varphi\in C^{\infty}(\overline{\Omega})$, we have
\[
\big|\int_{\Omega}\nabla (\sqrt{p}u_p)\nabla\varphi +
\sqrt{p}u_p\varphi\,dx \big|=
\sqrt{p}\big|\int_{\partial\Omega}|u_p|^{p-1}u_p\varphi\,
d\sigma(x)\big| \leq\frac{\|\varphi\|_{\infty}}{\sqrt{p}} p
\int_{\partial\Omega}|u_p|^{p}\,d\sigma(x)
\overset{(iv)}{\leq}\frac{\|\varphi\|_{\infty}}{\sqrt{p}} C
\]
for $p$ large. Hence
\[
\int_{\Omega}\nabla w\nabla\varphi+w\varphi\,dx=0\quad
\forall\varphi\in C^{\infty}(\overline{\Omega}),
\]
which implies that
 $w=0$ a.e. in $\Omega$.
\end{proof}
%%%%%%%%%%%%%%%%%%%%%%%%%%%%%%%%%%%%%%%%%%%%%%%%
%%%%%%%%%%%%%%%%%%%%%%%%%%%%%%%%%%%%%%%%%%%%%%%%%%%
\section{Proof of Theoerem \ref{theorem}}
%%%%%%%%%%%%%%%%%%%%%%%%%%%%%%%%%%%%%%%%%%%%%%%%%%%%%%%%%%%%%%%%%%%%%%%%%%%%%%%%%%%%%%%%

%%%%%%%%%%%%%%%%%%%%%%%%%%%%%%%%%%%%%%%%%%%%%%%%%%%%%%%
%%%%%%%%%%%%%%%%%%%%%%%%%%%%%%%%%%%%%%%%%%%%%%%%%%%%%%%%%%%%%%%%%%
%%%%%%%%%%%%%%%%%%%%%%%%%%%%%%%%%%%%%%%%%%%%%%%%%%
We start with the following interesting result contained in
\cite{DavilaDM}, which is a variant of an estimate of Brezis and
Merle \cite{Brezis-Merle}.
\begin{lemma}\label{BM}
Consider the linear equation
\begin{equation}\label{linear equation}
\left\{\begin{array}{lr}\Delta u= u\qquad  \mbox{ in }\Omega\\
\frac{\partial u}{\partial \nu}=h \qquad\mbox{ on }\partial \Omega
\end{array}\right.
\end{equation}
with $h\in L^1(\partial\Omega)$.\\
For any $0 < k < \pi$ there exists a constant $C$ depending on $k$
and $ \Omega$ such that for any $h \in L^1(\partial \Omega)$ and $u$
the solution of \eqref{linear equation} we have
$$\displaystyle\int_{\partial \Omega}\exp\left[\frac{k|u(x)|}{\|h\|_{L^1(\partial \Omega)}}\right]\ d\sigma(x) \leq C.$$
\end{lemma}
%\begin{lemma}
%There exists a constant $c > 0$ such that for all $h\in
%L^1(\partial\Omega)$ with $h\geq0$, the solution $u$ of
%\eqref{linear equation} satisfies
%$$u(x)\geq c\int_{\partial \Omega} h\ d\sigma(x) \quad \hbox{a.e. } \partial\Omega.$$
%\end{lemma}
Let $u_p$ be a family of positive solutions to \eqref{problem1}
satisfying \eqref{energybound}. We recall that
$v_p=u_p/\int_{\partial \Omega}u_p^p \ d\sigma(x)$ and $f_p=
u_p^p/\int_{\partial \Omega}u_p^p \ d\sigma(x)$. Hence $v_p$
satisfies
\begin{equation*}\label{vequation}
\left\{\begin{array}{lr}\Delta v_p= v_p\qquad  \mbox{ in }\Omega\\
\frac{\partial v_p}{\partial \nu}=f_p \qquad \mbox{ on }\partial
\Omega.
\end{array}\right.
\end{equation*}
%If $f_n$ is unbounded in $L^1(\partial\Omega)$ then by Lemma 9.2 we
%see that for a subsequence $v_p\rightarrow \infty$ uniformly in $\partial \Omega$ and by maximum principle this holds in $\Omega$ too.\\
We now define an important quantity:
\begin{equation}\label{L0}
\displaystyle L_0=\limsup_{p\rightarrow +\infty} \frac{p\
\gamma_p}{e}
\end{equation}
where $$\gamma_p=\int_{\partial\Omega}u_p^p \ d\sigma(x).$$ Note
that the quantity $L_0$ is well defined using \eqref{energybound}, \eqref{energiaSuSoluzioni} and H\"{o}lder inequality.\\
In the sequel, we denote any sequence $u_{p_n}$ of $u_p$ by $u_n$
and $\gamma_{p_n}$ of $\gamma_p$ by $\gamma_n$.\\Since $u_n$ has the
property
$$\int_{\partial \Omega}f_n \
 d\sigma(x)=\displaystyle\int_{\partial \Omega}\frac{u_n^{p_n}}{\int_{\partial \Omega}u_n^{p_n}\
 d\sigma(x)}\ d\sigma(x)=1$$
we can subtract a subsequence of $u_n$, still denoted by $u_n$, such
that there exists a positive bounded measure $\mu$ in $M(\partial
\Omega)$, the set of all real bounded Borel measures on $\partial
\Omega$, such that $\mu(\partial\Omega)\leq 1$ and
$$\displaystyle \int_{\partial \Omega}f_n\varphi\longrightarrow\displaystyle \int_{\partial \Omega}\varphi d\mu$$
for all $\varphi\in C(\partial \Omega)$ where
$$v_n = u_n/\gamma_n \quad \hbox{and} \quad f_n = \gamma_n^{p_n-1}v^{p_n}_n.$$
To analyze the measure $\mu$, we introduce some notations. For any
$\delta
> 0$, we call $x_0$ a $\delta$-regular point if there exists a
function $\varphi$ in $C(\partial \Omega)$, $0\leq \varphi \leq 1$,
with $\varphi= 1$ in a neighborhood of $x_0$ such that
\begin{equation}
\displaystyle\int_{\partial \Omega}\varphi \ d\mu<
\frac{\pi}{L_0+2\delta}.
\end{equation}
We define $$\Sigma(\delta) = \{x_o \in \partial \Omega : \quad x_o
\hbox{ is not a } \delta \hbox{-regular point}\}.$$ Our next lemma
plays a central role in the proof of Theorem \ref{theorem}. It says
that smallness of $\mu$ at a point $x_0$ implies boundedness of
$v_n$ near $x_0$.
\begin{lemma}\label{bound}
Let $x_0\in \partial \Omega$ be a $\delta$-regular point for some
$\delta> 0$. Then ${v_n}$ is bounded in $L^\infty(B_{R_0}(x_0)\cap
\Omega)$ for some $R_0> 0$.
\end{lemma}
\begin{proof}
Let $x_0$ be a regular point. From the definition of regular points,
there exists $R > 0$ such that
$$\displaystyle\int_{\partial \Omega\cap B_{R}(x_0)}f_n \ d\sigma(x) <
\frac{\pi}{L_0+\delta}.$$ Put $a_n = \chi_{B_{R}(x_0)}f_n$ and $b_n
= (1-\chi_{B_{R}(x_0)})f_n$ where $\chi_{B_{R}(x_0)}$ denotes the
characteristic function of $B_R(x_0)$. Split $v_n = v_{1n} +
v_{2n}$, where $v_{1n}$ and $v_{2n}$ are  solutions to
$$\left\{\begin{array}{lr}\Delta v_{1n}= v_{1n}\qquad  \mbox{ in }\Omega\\
\frac{\partial v_{1n}}{\partial \nu}=a_n \mbox{ on }\partial \Omega
\end{array}\right.
$$ and
$$\left\{\begin{array}{lr}\Delta v_{2n}= v_{2n}\qquad  \mbox{ in }\Omega\\
\frac{\partial v_{2n}}{\partial \nu}=b_n \mbox{ on }\partial \Omega
\end{array}\right.
$$ respectively.\\
By the maximum principle, we have $v_{1n}, \  v_{2n} > 0$. Since
$b_n = 0$ on $B_{R}(x_0)$, elliptic estimates imply that
\begin{equation}\label{v2n}\|v_{2n}\|_{L^{\infty}(B_{R/2}(x_0)\cap\Omega)}\leq
C\|v_{2n}\|_{L^{1}(B_{R}(x_0)\cap\Omega)}\leq C,
\end{equation}
where we used the fact $\|v_{2n}\|_{L^{1}(\Omega)} =\|\Delta
v_{2n}\|_{L^{1}(\Omega)} = \|b_{n}\|_{L^{1}(\partial\Omega)}\leq C$
for the last inequality.
Thus we have to consider $v_{1n}$ only.\\
{\bf{Claim}}: There exists some $q>1$ such that
$$\displaystyle \int_{B_{R/2}(x_0)\cap \partial \Omega}f_n^q \ d\sigma(x)\leq C.$$
Indeed, let $t$ be such that $t' = L_0 + \delta/2$ where $t'$ is the
H\"{o}lder conjugate of $t$. From
$$\displaystyle\int_{\partial \Omega\cap B_{R}(x_0)}f_n \ d\sigma(x) <
\frac{\pi}{L_0+\delta}$$ using Lemma \ref{BM}, we have
\begin{equation}\label{(4.6)}
\displaystyle\int_{\partial \Omega\cap
B_{R}(x_0)}\exp\bigg[(L_0+\delta/2)|v_{1n}|\bigg]  \ d\sigma(x) \leq
C.
\end{equation}
Now observe that $\log (x) \leq x/e$ for $x > 0$. As in
\cite{RenWeiTAMS1994} (see page $759$), we have
$$p_n\log \frac{u_n}{\gamma_n^{1/p_n}}\leq t' \frac{u_n}{\gamma_n}$$
for $n$ large enough because $\displaystyle\lim_{n\rightarrow\infty}
\gamma_n^{1/p_n}=1$ which follows from \eqref{boundEnergiap}. Hence
$$f_n \leq e^{t'v_n}, \quad (f_n)^t e^{-tv_{1n}}\leq e^{(t'+t)v_{2n}+t'v_{1n}}.$$
Therefore since $v_{2n}$ is uniformly bounded on $B_{R/2}(x_0)\cap
\Omega$, we have
\begin{equation}\label{(4.7)}
(f_n)^t e^{-t v_{1n}} \leq Ce^{t'v_{1n}}\quad\hbox{on
}B_{R/2}(x_0)\cap \partial \Omega.
\end{equation}
Combining \eqref{(4.6)} and \eqref{(4.7)}, we get that
$f_ne^{-v_{1n}}$ is bounded in $L^t(B_{R/2}(x_0)\cap \partial
\Omega)$.\\
Fix $\eta > 0$ small enough such that $\pi-\eta >
\frac{\pi}{L_0+\delta}(t'+\eta)$. By Lemma \ref{BM} we have
$$\displaystyle\int_{\partial \Omega\cap
B_{R}(x_0)}\exp[(t'+\eta)|v_{1n}|]  \ d\sigma(x) \leq C.$$ Therefore
$e^{v_{1n}}$ is bounded in $L^{t'+\eta}(\partial \Omega\cap
B_{R}(x_0))$ and so $f_n=f_ne^{-v_{1n}}.e^{v_{1n}}$ is bounded in
$L^{q}(\partial \Omega\cap B_{R/2}(x_0))$ for some $q >
1$. Hence we get the claim.\\
This fact and elliptic estimates imply that $v_{1n}$ is uniformly
bounded in $L^\infty( \Omega\cap B_{R/4}(x_0))$. Taking account of
\eqref{v2n} and choosing $R_0= R/4$ the desired result follows.
\end{proof}
%{\bf{Proof of Theorem \ref{theorem}}}:
Let's go back to the proof of Theorem \ref{theorem}. Taking account
of \eqref{boundEnergiap} and Lemma \ref{bound}, by the same argument
of Ren and Wei (see \cite{RenWeiTAMS1994} page $759$), we have $S =
\Sigma(\delta)$ for any $\delta
>0$.
%\begin{lemma}
% Let $S$ be the blow-up set defined in $(1.5)$ of the subsequence
%$v_n$. Then $S$ is nonempty and $S \subset\partial \Omega$.
%\end{lemma}
We get $S=\{x_o \in \partial \Omega : \quad x_o \hbox{ is not a }
\delta \hbox{-regular point for any }\delta>0\}$. Then
\begin{equation}\label{nnregular}\mu(\{x_0\})\geq \frac{\pi}{L_0+2\delta}
\end{equation}
for all $x_0 \in S$ and for any $\delta > 0$.\\
Hence $S$ is a finite nonempty set (since $\mu(\Omega)\leq 1$ and
$\|v_n\|_{L^\infty(\partial \Omega)}\rightarrow +\infty$) and from
Lemma \ref{bound} for every $x \in\partial \Omega\setminus S$ we
have that $v_n$ is bounded in a neighborhood of $x$. Then $v_n$ is
bounded in compact subsets of $\partial \Omega\setminus S$ and so
$f_n\rightarrow 0$ uniformly on compact subsets of $\partial
\Omega\setminus S$ using \eqref{boundEnergiap}. This shows that the
support of $\mu$ is contained in $S$ and therefore we can write
\begin{equation}\label{measure}\displaystyle\mu=\sum^m_{i=1}a_i\delta_{x_i}
\end{equation}
where $a_i > 0$ and $x_i\in\partial \Omega$. Hence we get parts
$(1)$ and $(2)$ and we will come back later to the proof of
\eqref{weight}. We point out that \eqref{nnregular} and
\eqref{measure} imply that $a_i\geq \pi/ L_0$.\\ Now, we need the
following elliptic $L^1$ estimate by Brezis and Strauss
\cite{Brezis-Strauss} for weak solutions with the $L^1$ Neumann
data.
\begin{lemma}\label{brezis-strauss}
Let $u$ be a weak solution of
\begin{equation*}
\left\{\begin{array}{lr}-\Delta u+ u= f  \ \mbox{ in }\Omega,\\
\frac{\partial u}{\partial \nu}=h \qquad \quad \mbox{ on }\partial
\Omega
\end{array}\right.
\end{equation*}
with $f \in L^1(\Omega)$ and $g \in L^1(\partial \Omega)$, where
$\Omega$ is a smooth bounded domain in $\mathbb{R}^N$, $N \geq 2$.
Then we have $u \in W^{1,q}(\Omega)$ for all $1 \leq q < \frac{N}{
N-1}$ and
$$\|u\|_{W^{1,q}(\Omega)}\leq C_q(\|f\|_{L^1(\Omega)}+\|g\|_{L^1(\partial\Omega)})$$
holds.
\end{lemma}
Using Lemma \ref{brezis-strauss}, we have $v_n$ is uniformly bounded
in $W^{1,q}(\Omega)$ for any $1\leq q < 2$. Thus, by choosing a
subsequence, we have a function $v^*$ such that $v_n\rightharpoonup
v^*$ weakly in $W^{1,q}(\Omega)$ for any $1 \leq q < 2$, $v_n
\rightarrow v^*$ strongly in $L^t(\Omega)$ and $L^t(\partial
\Omega)$ respectively for any $1\leq t < \infty$. The last
convergence follows by the compact embedding $W^{1,q}(\Omega)
\hookrightarrow L^t(\Omega)$ for any $1  \leq t < q /(2-q) $. Thus
by taking the limit in the equation
$$\int_{\Omega}(-\Delta \varphi+\varphi)v_n \ dx =
\int_{\partial \Omega}f_n\varphi \ d\sigma(x) -\int_{\partial
\Omega}\frac{\partial \varphi}{\partial \nu} v_n\ d\sigma(x)$$ for
any $\varphi\in C^1(\overline{\Omega})$, we obtain
$$\int_{\Omega}(-\Delta \varphi+\varphi)v^* \ dx +\int_{\partial
\Omega}\frac{\partial \varphi}{\partial \nu} v^*\ d\sigma(x) =
\sum^m_{i=1}a_i\varphi(x_i)$$ which implies $v^*$ is the solution of
the following problem
$$\left\{\begin{array}{lr}\Delta v^*= v^*\qquad  \mbox{ in }\Omega\\
\displaystyle\frac{\partial v^*}{\partial
\nu}=\sum^m_{i=1}a_i\delta_{x_i} \mbox{ on }\partial \Omega
\end{array}\right.
$$
From this it follows that
$$\displaystyle v^*(x)=\sum_{i=1}^m a_iG(x,x_i).$$
In the sequel, we will prove that $\displaystyle v_{n}\rightarrow
v^*$ in $C^1_{loc}(\overline{\Omega}\setminus S)$. We start by using
Green representation for $v_n$:
\begin{equation}
v_n(x)=\displaystyle\int_{\partial\Omega}G(x,y)f_n(y)\ d\sigma(y),
\end{equation}
where $G(x,y)$ is Green's function for Neumann problem
\eqref{Greenequation}. Suppose $x\in\Omega$ and
$d=dist(x,\partial\Omega)$. Then, for $z\in B_{d/2}(x),$ we have
$dist(z,\partial\Omega)\geq\frac{1}{2}d,$ and
\begin{eqnarray}\label{interiorbound}
|u_n(z)|&\leq& \displaystyle\int_{\partial\Omega} |\frac{1}{\pi}\log
\frac{1}{ |z - y|}+H(z, y)| f_n(y)\ d\sigma(y)\nonumber\\
&\leq& \int_{\partial\Omega} \left(\frac{1}{\pi}|\log \frac{2}{
d}|+|H(z, y)|\right) f_n(y)\ d\sigma(y)\nonumber\\
 &\leq&C(|\log d|+1)\int_{\partial\Omega}f_n(y)d\sigma(y)\leq C,
 \quad \forall z\in B_{d/2}(x).
\end{eqnarray}
Let $K$ be a compact set in $\overline{\Omega}\setminus S$. From
Lemma \ref{bound} and \eqref{interiorbound}, we have that
$$v_n\leq C \hbox{ on }K.$$
Since $v_n$ are bounded on $K$ and satisfy $\Delta v_n -v_n=0$ in
$\mathring{K}$, we have by the elliptic regularity theory a
subsequence of $v_n$, still denoted by $v_n$ that approaches the
same function $v^*$ in $C^1(K)$.\\
We proved part $(3)$.\\

Finally, we prove simultaneously \eqref{weight} and Statement $(4)$
of Theorem \ref{theorem}, that are the choice of the weights $a_i$'s
and the localization of the concentration points. We borrow the idea
of \cite{DavilaDM} and derive Pohozaev-type identities in balls
around the peak point. Let us concentrate on $x_1$. Without loss of
generality, We may assume $x_1 = 0$. In the sequel, we use a
particular straightening of the boundary introduced in
\cite{DavilaDM}. That is a conformal diffeomorphism $\Phi_c: \Pi
\cap B_{R_1} \longrightarrow \Omega \cap B_{r_1}$ which flattens the
boundary $\partial \Omega$, where $\Pi = \{(y_1, y_2)\ |\ y_2
> 0\}$ denotes the upper half space and $R_1
> 0$ is a radius sufficiently small such that $(\partial\Omega \cap B_{r_1})\cap S=\{0\}$. We may choose $\Phi_c$ is at least $C^3$, up to $\partial \Pi\cap B_{R_1}$, $\Phi_c(0) = 0$ and $D\Phi_c(0) = Id$. Set $\widetilde{u}_n(y)
= u_n(\Phi_c(y))$ for $y = (y_1, y_2) \in  \Pi \cap B_{R_1}$. Then
by the conformality of $\Phi_c$, $\widetilde{u}_n$
 satisfies
 \begin{equation}\label{pb}
\left\{\begin{array}{lcl}-\Delta \widetilde{u}_n+b(y)\widetilde{u}_n= 0\qquad  &\mbox{in}&\Pi \cap B_{R_1},\\
\frac{\partial \widetilde{u}_n}{\partial
\nu}=h(y)\widetilde{u}_n^{p_n}\ &\mbox{on}&\partial \Pi \cap
B_{R_1},
\end{array}\right.
\end{equation}
where $\widetilde{\nu}$ is the unit outer normal vector to
$\partial(\Pi \cap B_{R_1} )$, $b$ and $h$ are defined as
$$b(y) = |detD\Phi_c(y)|,\quad  h(y) = |D\Phi_c(y)e|$$
with $e = (1,0)$. Note that $\widetilde{\nu}(y) = \nu(\Phi_c(y))$
for $y \in
\partial \Pi\cap B_{R_1} $. Note also that, by using a clever idea of
\cite{DavilaDM}, we can modify $\Phi_c$ to prescribe the number
$$\alpha =\frac{\big(\frac{\partial h}{\partial y_1}\big)}{h(y)^2}\bigg|_{y=0}=
\bigg(\frac{\partial h}{\partial y_1}\bigg)(0).$$ Let $R$ be such
that $0<R<R_1$. Applying now the Pohozaev identity to problem
\eqref{pb} we get
\begin{eqnarray*}\label{pohozaevidentity1}
&&\displaystyle\int_{\Pi\cap
B_R}b(y)\widetilde{u}_n^2(y)dy+\frac{1}{2}\int_{\Pi\cap
B_R}(y-y_0,\nabla
b(y))\widetilde{u}_n^2(y)dy\\&&=\frac{1}{2}\int_{\partial (\Pi\cap
B_R)}
(y-y_0,\widetilde{\nu})b(y)\widetilde{u}_n^2(y)d\sigma(y)-\int_{\partial
(\Pi\cap B_R)} (y-y_0,\nabla \widetilde{u}_n(y))\frac{\partial
\widetilde{u}_n}{\partial
\nu}d\sigma(y)\\&&+\frac{1}{2}\int_{\partial (\Pi\cap
B_R)}(y-y_0,\widetilde{\nu})|\nabla \widetilde{u}_n|^2 d\sigma(y)
\hbox{ for any }y_0\in \mathbb{R}^2,
\end{eqnarray*}
where and from now on, $\widetilde{\nu}$ will be used again to
denote the unit normal to $\partial(H \cap B_R)$. The proof of the
Pohozaev identity is standard and it is omitted here.
Differentiating with respect to $y_0$, we have, in turn,
\begin{eqnarray*}
&&\int_{\partial (\Pi\cap B_R)} \nabla
\widetilde{u}_n(y)\frac{\partial \widetilde{u}_n}{\partial
\widetilde{\nu}}\ d\sigma(y)\\&&=\frac{1}{2}\int_{\partial (\Pi\cap
B_R)} \big(|\nabla \widetilde{u}_n|^2+b(y)\widetilde{u}_n^2\big)
\widetilde{\nu}d\sigma(y)-\frac{1}{2}\int_{ \Pi\cap B_R} \nabla
b(y)\widetilde{u}_n^2\  dy.
\end{eqnarray*}
Since $\widetilde{\nu}=
(\widetilde{\nu}_1,\widetilde{\nu}_2)=(0,-1)$ on $\partial \Pi \cap
B_R$, the first component of the above vector equation reads
\begin{eqnarray}\label{46}
&&\int_{\partial \Pi\cap B_R}( \widetilde{u}_n)_{y_1}h(y)
\widetilde{u}_n^{p_n}d\sigma(y)+\int_{\Pi\cap \partial B_R} (
\widetilde{u}_n)_{y_1}\frac{\partial \widetilde{u}_n}{\partial
\widetilde{\nu}}\ d\sigma(y)\nonumber\\
&&=\frac{1}{2}\int_{\Pi\cap \partial B_R} \big(|\nabla
\widetilde{u}_n|^2+b(y)\widetilde{u}_n^2\big)
\widetilde{\nu}_1d\sigma(y)-\frac{1}{2}\int_{ \Pi\cap B_R}
b_{y_1}(y)\widetilde{u}_n^2 dy,
\end{eqnarray}
where $( )_{y_1}$ denotes the derivative with respect to $y_1$.\\
Recall that $\gamma _n = \int_{\partial \Omega} u^{p_n}_ n
d\sigma(x)$. From the fact that $\widetilde{f}_n(y) =
\frac{\widetilde{u}^{p_n}_ n}{\gamma_n}\rightharpoonup a_1\delta_0$
in the sense of Radon measures on $\partial \Pi \cap B_R$,
\eqref{boundSoluzio} and \eqref{boundEnergiap},
%$\|\widetilde{u}_n\|_{L^{\infty}(\partial \Pi\cap B_R) }= O(1)$ uniformly in $n$
 we see
$$\widetilde{g}_n(y):=\frac{1}{\gamma_n^2}\frac{\widetilde{u}_n^{p_n+1}(y)}{p_n+1}=
\frac{1}{(p_n+1)\gamma_n}\widetilde{f}_n(y)\widetilde{u}_n(y)$$
satisfies that $supp(\widetilde{g}_n)\rightarrow\{0\}$ and
$\int_{\partial \Pi\cap B_R}\widetilde{g}_nd\sigma(y) = O(1)$ as
$n\rightarrow +\infty$. Thus, by choosing a subsequence, we have the
convergence
\begin{equation}
\widetilde{g}_n(y)=\frac{1}{\gamma_n^2}\frac{\widetilde{u}_n^{p_n+1}(y)}{p_n+1}\rightharpoonup
C_1\delta_0
\end{equation}
in the sense of Radon measures on $\partial \Pi\cap B_R$, where $C_1
= lim_{n\rightarrow +\infty} \int_{\partial \Pi\cap B_R}
\widetilde{g}_nd\sigma(y)$ (up to a subsequence). By using this
fact, we have
\begin{eqnarray*}
&&\frac{1}{\gamma_n^2}\int_{\partial \Pi\cap B_R}(
\widetilde{u}_n)_{y_1}h(y)
\widetilde{u}_n^{p_n}d\sigma(y)\\&&=\bigg[\frac{h(y)}{\gamma
_n^2}\frac{\widetilde{u}_n^{p_n+1}(y)}{p_n+1}\bigg]^{y_1=R}_{y_1=-R}-\int_{\partial
\Pi\cap B_R}h_{y_1}(y)\frac{\widetilde{u}_n^{p_n+1}(y)}{(p_n+1)\gamma_n^2}d\sigma(y)\\
&&\rightarrow 0-C_1h_{y_1}(0)=-C_1 \alpha
\end{eqnarray*}
as $n\rightarrow +\infty$.\\
Let $\widetilde{v}^*  (y) = v^* ( \Phi_c (y))$ denote the limit
function in the $y$ coordinates, and observe that
$$\widetilde{v}_{p_n}
\rightarrow \widetilde{v}^* \hbox{ in } C^1_{ loc}(\overline{\Pi
}\cap B_R\backslash\{0\} ).$$
 Thus after
dividing \eqref{46} by $\gamma_n^2$ and then letting $n\rightarrow
\infty$, we obtain
\begin{eqnarray}\label{47}
&&-C_1\alpha +\int_{ \Pi\cap \partial B_R}
\widetilde{v}^*_{y_1}\frac{\partial \widetilde{v}^*}{\partial
\widetilde{\nu}}\ d\sigma(y)\nonumber\\
&&=\frac{1}{2}\int_{\Pi\cap \partial B_R} \big(|\nabla
\widetilde{v}^*|^2+b(y)(\widetilde{v}^*)^2\big) \widetilde{\nu}_1\
d\sigma(y)-\frac{1}{2}\int_{ \Pi\cap B_R}
b_{y_1}(y)(\widetilde{v}^*)^2\ dy.
\end{eqnarray}
%where $\widetilde{G}(y) = G(\Phi_c(y), 0)$ is a limit function of
%$\widetilde{v}_n(y) = v_n(\Phi_c(y)) =
%\frac{\widetilde{u}_n(y)}{\gamma_n}$.
At this point, we have the same formula as the equation $(117)$ in
\cite{DavilaDM}, thus we obtain the result. Indeed decompose
$v^*(x)=s(x)+w(x)$ where
$$s(x)=\frac{a_1}{\pi}\log \frac{1}{|x|},\quad w(x)=a_1H(x,0)+\sum_{j=2}^ma_jG(x,x_j). $$
We define then the corresponding functions in the new coordinates
$$\widetilde{s}(y) = s( \Phi_c (y)),\quad \widetilde{w}(y) = w( \Phi_c  (y)),\quad y\in  \Pi\cap B_{R_1} .$$
Using this decomposition and the fact that $\widetilde{w}$ satisfies
$$-\Delta \widetilde{w}+b(y)\widetilde{w}=-b(y)\widetilde{s}(y)\quad \hbox{in }\Pi \cap B_{R_1},$$
we derive from \eqref{47}
\begin{eqnarray}\label{48}
&&-C_1\alpha +\int_{ \Pi\cap \partial
B_R}\widetilde{s}_{\widetilde{\nu}}
\widetilde{s}_{y_1}+\widetilde{s}_{\widetilde{\nu}}
\widetilde{w}_{y_1}+\widetilde{s}_{y_1}\widetilde{w}_{\widetilde{\nu}}\
d\sigma(y)\nonumber\\
&&=\int_{\Pi\cap \partial B_R} \big(\frac{1}{2}|\nabla
\widetilde{s}|^2+\nabla \widetilde{s}\nabla
\widetilde{w}+\frac{1}{2} b(y)\widetilde{s}^2+b(y)
\widetilde{s}\widetilde{w}\big)
\widetilde{\nu}_1d\sigma(y)\nonumber\\
&&-\int_{  \Pi\cap B_R}
b_{y_1}(y)(\frac{1}{2}\widetilde{s}^2+\widetilde{s}\widetilde{w})\
dy+\int_{
\partial \Pi\cap B_R}\widetilde{w}_{\widetilde{\nu}}
\widetilde{w}_{y_1}\ d\sigma(y)\nonumber\\
&&- \int_{  \Pi\cap B_R}b(y)\widetilde{s}\widetilde{w}_{y_1}\  dy,
\end{eqnarray}
where $\widetilde{s}_{\widetilde{\nu}}$ and
$\widetilde{w}_{\widetilde{\nu}}$
 are the partial
derivatives with respect to $\widetilde{\nu}$ of the functions
$\widetilde{s}$ and $\widetilde{w}$, respectively. Letting
$R\rightarrow 0$ and using \cite[Lemma 9.3]{DavilaDM} together with
\eqref{48} we obtain
$$-\alpha C_1 + \frac{3\alpha}{4\pi}a_1^2- a_1\widetilde{w}_{y_1}(0)=
\frac{\alpha}{4\pi}a_1^2- \frac{a_1}{2}\widetilde{w}_{y_1}(0)$$ that
is
$$\alpha \big(\frac{a_1^2}{2\pi}-C_1\big)=\frac{1}{2}a_1\widetilde{w}_{y_1}(0).$$
Since $\alpha\in \mathbb{R}$ can be chosen arbitrarily, we conclude
that $$\widetilde{w}_{y_1}(0) = 0\quad \hbox{and}\quad  a_1^2=2\pi
C_1=2\pi \displaystyle\lim_{R\rightarrow 0}\lim_{n\rightarrow
+\infty} \int_{\partial \Pi\cap B_R} \widetilde{g}_nd\sigma(y).$$
Consequently
the desired conclusion of Theorem \ref{theorem} (4) follows and by using a change of variables we get \eqref{weight}. \cqfd\\

{\bf{Acknowledgment:}}\\
The author wishes to express his sincere gratitude to Professor
Filomena Pacella for her kind hospitality and support during his
visits to the Department of Mathematics of the University of Roma
"La Sapienza". This work was partially supported by the " Istituto
Nazionale di Alta Matematica ". He also thanks Professors Franchesca
De Marchis and Isabella Ianni for valuable discussions.
%%%%%%%%%%%%%%%%%%%%%%%%%%%%%%%%%%%%%%%%%%%%%%%%%%%%%%%%%%%%%%%%%%%%%%%%%%%%%%%%%%%
%%%%%%%%%%%%%%%%%%%%%%%%%%%%%%%%%%%%%%%%%%%%%%%%%%%%%%%%%%%%%%%%%%%%%%%%%%%%%%%%%%%%%%%%%%
%%%%%%%%%%%%%%%%%%%%%%%%%%%%%%%%%%%%%%%%%%%%%%%%%%%%%%%%%%%%%%%%%%%%%%%%%%


\begin{thebibliography}{99}

\bibitem{AdiGrossi} Adimurthi, M. Grossi, \emph{Asymptotic estimates  for a two-dimensional problem with polynomial nonlinearity}, Proc. Amer. Math. Soc. 132 (2004), no. 4, 1013-1019.
%

%
\bibitem{BFS} M. Ben Ayed, H. Fourti, A. Selmi, \emph{Harmonic Functions with nonlinear Neumann
boundary condition and their Morse indices}, Nonlinear Analysis:
Real World Applications 38 (2017) 96-112.
%

%
\bibitem{Brezis-Merle} H. Brezis, F. Merle, \emph{Uniform estimates and blow-up behavior for
solutions of $-\Delta u = V (x)e^u$ in two dimensions}, Comm.
Partial Differential Equations, 16 (1991) 1223-1253.
%
\bibitem{Brezis-Strauss} H. Brezis, W. Strauss, \emph{Semi-linear second-order elliptic
equations in $L^1$}, J. Math. Soc. Japan, 25 (1973) 565-590.



\bibitem{Castro2} H. Castro, \emph{Solutions with spikes at the boundary for a 2D nonlinear
Neumann problem with large exponent}, J. Differential Equations, 246
(2009) 2991-3037.

\bibitem{Castro1} H. Castro, \emph{Asymptotic estimates for the least energy solution of a planar semi-linear
Neumann problem,}

%
\bibitem{Cherrier} P. Cherrier, \emph{Probl\`emes de Neumann non lin\'eaires sur
les varietes riemanniennes}, J. Functional Anal., 57, 154-206 (1984)

\bibitem{DavilaDM} \emph{J. Davila, M. del Pino, M. Musso}, Concentrating solutions in
a two-dimensional elliptic problem with exponential Neumann data, J.
Functional Anal., 227 (2005) 430-490.


\bibitem{DFIP} F. De Marchis, H. Fourti, I. Ianni, F. Pacella,
\emph{to appear}.

\bibitem{DGIP1}
F. De Marchis, M. Grossi, I. Ianni and F. Pacella,
\emph{$L^\infty$-norm and energy quantization for the planar
Lane-Emden problem with large exponent}, Archiv der Mathematik, 111
(4) 421-429 (2018).

\bibitem{DeMarchisIanniPacellaJEMS} F. De Marchis, I. Ianni, F. Pacella,
\emph{Asymptotic analysis and sign changing bubble towers for
Lane-Emden problems}, Journal of the European Mathematical Society
17 (8) (2015) 2037-2068.

\bibitem{DeMarchisIanniPacellaPositivesolutions} F. De Marchis, I. Ianni, F. Pacella, \emph{Asymptotic profile of
positive solutions of Lane-Emden problems in dimension two}, J. Fix.
Point Theory A. 19 (2017), no.1, 889-916.


%
\bibitem{DeMarchisIanniPacellaLondon}
F. De Marchis, I. Ianni and F. Pacella, \emph{Asymptotic analysis of
the Lane-Emden problem in dimension two, contribution in the volume
Partial Differential Equations arising from Physics and Geometry},
London Mathematical Society Lecture Note Series (No. 450), Cambridge
University Press (2018).
%
\bibitem{Druet} O. Druet, \emph{Multibumps analysis in dimension 2: quantification of blow-up levels}, Duke Math. J. 132 (2006), no. 2, 217-269.
%
\bibitem{E}
L. C. Evans, \emph{Partial differential equations}. Americain
Mathematical society, (1998). Third  printing, 2002.
%
\bibitem{GT} Gilbarg D, Trudinger N S. \emph{Elliptic partial differential equation of second order}.
Springer-Verlag, 1977
%
\bibitem{GL}
Y. X. Guo, J. Q. Liu, \emph{Blow-up analysis for solutions of the
Laplacian equation with exponential Neumann boundary condition in
dimension two},Commun. Contemp. Math., 8,  (2006)737-761.
%
\bibitem{GrossiGrumiauPacella1} M. Grossi, C. Grumiau, F. Pacella,
\emph{Lane Emden problems: asymptotic behavior of low energy nodal
solutions}, Ann. Inst. H. Poincar\'{e} Anal. Non Lin\'{e}aire 30
(2013), no. 1, 121-140.
%

%
\bibitem{LWZ} A. Le, Z-Q. Wang, J. Zhou, \emph{Finding Multiple Solutions to Elliptic PDE with Nonlinear
Boundary Conditions}, J. Sci. Comput. 56 (2013), 591-615.
%
\bibitem{Liu}
P. Liu, \emph{A Moser-Trudinger type inequality and blow-up analysis
on Riemann surfaces}, dissertation, Der Fakultat fur Mathematik and
Informatik der Universitat Leipzig (2001).
%


\bibitem{RenWeiTAMS1994} X. Ren, Xiaofeng, J. Wei, \emph{On a two-dimensional elliptic problem with large exponent in nonlinearity}, Trans. Amer. Math. Soc. 343 (1994), no. 2, 749-763.

\bibitem{RenWeiPAMS1996}
X. Ren, Xiaofeng, J. Wei, \emph{Single-point condensation and
least-energy solutions}, Proc. Amer. Math. Soc. 124 (1996), no. 1,
111-120.

\bibitem{RR}
M. Renardy and R. Rogers, \emph{Introduction to partial differential
equations,} Springer, Berlin, (1993).

\bibitem{SantraWei}
S. Santra, J. Wei,  \emph{Asymptotic behavior of solutions of a
biharmonic Dirichlet problem with large exponents}, J. Anal. Math.
115 (2011), 1-31.

\bibitem{Shamir1} E. Shamir, \emph{Mixed boundary value problems for elliptic equations
in the plane. The $L^p$ theory}, Ann. Scuola Norm. Sup. Pisa (3) 17
(1963), 117-139.






\bibitem{Takahashi} F. Takahashi, \emph{Asymptotic behavior of least energy solutions
for a 2D nonlinear Neumann problem with large exponent}, J. Math.
Anal. Appl. 411 (2014), no. 1, 95-106.


\bibitem{Thizy}
P.-D. Thizy, \emph{Sharp quantization for Lane-Emden problems in
dimension two,} Pacific Journal of Mathematics 300 (2019), No. 2,
491-497.





\end{thebibliography}
\end{document}